\documentclass[final]{siamltex}

\RequirePackage{hypernat}%
\RequirePackage{amsmath}%
\RequirePackage{amsxtra}%
\RequirePackage{amstext}%
\RequirePackage{amssymb}%
\RequirePackage{latexsym}%
\RequirePackage{dsfont}%

\RequirePackage{color}%
\RequirePackage[colorlinks,citecolor=blue,urlcolor=blue,pagebackref]{hyperref}

\definecolor{backcol1}{rgb}{0.95,1,1}%
\definecolor{backcol2}{rgb}{1,0.95,1}%
\definecolor{backcol3}{rgb}{1,1,0.95}%
\definecolor{backcol4}{rgb}{1,0.95,0.95}%
\definecolor{backcol5}{rgb}{0.95,0.95,1}%
\definecolor{backcol6}{rgb}{0.95,1,0.95}%
\definecolor{backcol7}{rgb}{0.95,0.9,0.75}%

\definecolor{hellgelb}{rgb}{1,1,0.85}%
\definecolor{colKeys}{rgb}{0,0,1}%
\definecolor{colIdentifier}{rgb}{0,0,0}%
\definecolor{colComments}{rgb}{1,0,0}%
\definecolor{colString}{rgb}{0,0.5,0}%

\definecolor{red}       {rgb}{0.0,0.0,1.0}        
\definecolor{magenta}   {rgb}{0.0,0.0,1.0}        
\definecolor{cyan}      {rgb}{0.0,0.0,1.0}        
\definecolor{green}     {rgb}{0.0,0.4,0.3}        
\graphicspath{{./figs/}}%

\title{Active Set Algorithm for Large-Scale Continuous Knapsack Problems
with Application to Topology Optimization Problems}

\author{R. Tavakoli \thanks{Department of Material Science and Engineering, Sharif University of
Technology, Tehran, Iran, P.O. Box 11365-9466, email:
\href{mailto:tav@mehr.sharif.edu}{tav@mehr.sharif.edu}, %
\href{mailto:rohtav@gmail.com}{rohtav@gmail.com} %
 (early version of this paper is available on optimization-online server).%
 } }

\begin{document}

\maketitle

\begin{center}
 {{\tiny 11 October 2009}}\vspace*{5mm}%
\end{center}


\maketitle


\begin{abstract}%
The structure of many real-world optimization problems includes
minimization of a nonlinear (or quadratic) functional subject to
bound and singly linear constraints (in the form of either equality
or bilateral inequality) which are commonly called as continuous
knapsack problems. Since there are efficient methods to solve
large-scale bound constrained nonlinear programs, it is desirable to
adapt these methods to solve knapsack problems, while preserving
their efficiency and convergence theories. The goal of this paper is
to introduce a general framework to extend a box-constrained
optimization solver to solve knapsack problems. This framework
includes two main ingredients which are O(n) methods; in terms of
the computational cost and required memory; for the projection onto
the knapsack constrains and the null-space manipulation of the
related linear constraint. The main focus of this work is on the
extension of Hager-Zhang active set algorithm (SIAM J. Optim. 2006,
pp. 526--557). The main reasons for this choice was its promising
efficiency in practice as well as its excellent convergence theories
(e.g., superlinear local convergence rate without strict
complementarity assumption). Moreover, this method does not use any
explicit form of second order information and/or solution of linear
systems in the course of optimization which makes it an ideal choice
for large-scale problems. Moreover, application of Birgin and
Mart{\'\i}nez active set algorithm (Comp. Opt. Appl. 2002, pp.
101--125) for knapsack problems is also briefly commented. The
feasibility of the presented algorithm is supported by numerical
results for topology optimization problems.

\end{abstract}%






\section{Introduction}
\label{sec:int}

The goal of this paper is to develop efficient methods to solve
large-scale linearly constrained optimization problems with the
following structure
\begin{equation}%
\label{eq:P}
    \min f(\mathbf{x}) \quad \texttt{s.t.} \quad {\mathbf{x}} \in \mathcal{D}
\end{equation}
the feasible set $\mathcal{D}$ is equal to either of
$\mathcal{D}_{\mathcal{E}}$ or $\mathcal{D}_{\mathcal{I}}$ which are
defined as follows
\begin{align}%
\label{eq:deq}
   \mathcal{D}_{\mathcal{E}} \stackrel{def}{=}& \ \{ \mathbf{x} \in  \mathcal{B}: \quad
  \mathbf{a}^T \mathbf{x} = b \ \}
  \\ %
\label{eq:dieq}
   \mathcal{D}_{\mathcal{I}}  \stackrel{def}{=}& \ \{ \mathbf{x} \in  \mathcal{B}:  \quad
  b_l \leqslant \mathbf{a}^T \mathbf{x} \leqslant b_u \}
  \\%
\label{eq:box}%
   \mathcal{B} \stackrel{def}{=}& \ \{ \mathbf{x} \in  \mathbb{R}^n: \quad
  \mathbf{l} \leqslant \mathbf{x} \leqslant \mathbf{u}\}
\end{align}
where $f$ is a real-valued continuously differentiable function
defined on $\mathcal{D}$, $\mathbf{a}, \mathbf{l}, \mathbf{u} \in
\mathbb{R}^n$, $b, b_l, b_u \in \mathbb{R}$ and $\mathbf{l}
\leqslant \mathbf{u}$. In the related literature, set $\mathcal{D}$
is usually called as the continuous knapsack constraints. Although
problem \eqref{eq:P} can be studied under context of general
linearly constrained optimization problems, due to the importance
and wide range of applications, there are a lot of works
specifically conducted to solve problems with \eqref{eq:P}
sprecture, e.g. see:
\cite{more1990solution,melman2000efficient,dai2006new,dahiya2007convex,lucidi2007convergent,lin2009decomposition,tseng2008coordinate,gonzalez2009aas}.
Optimization problems like \eqref{eq:P} occur frequently in the
filed of operational researchs like optimal resource allocation and
marketing, refer to
\cite{bretthauer2002nonlinear,patriksson2008survey} as some topical
reviews.

In addition to the mentioned applications, finite-dimensional
counterpart of some constrained infinite-dimensional optimization or
control problems leads to a problem with structure like
\eqref{eq:P}. Topology optimization problems with resource
constraint can be accounted in this context. Note that, in this case
we have also a PDE as an equality constraints. However it is easy to
remove these constraints by performing the optimization on the
reduced space spanned by the state PDE.

Although in the past two decades, several methods with good
convergence theories have been introduced to solve linearly
constrained optimization problems, developing an efficient method to
solve large scale problems (in terms of computational cost and
memory usage) is still remained as an open problem. An O(n) growth
of the memory usage and computational cost per iteration as well as
the global convergence and quadratic or superlinear local
convergence seems to be desirable properties of an efficient method.
To realize this goal it is required to avoid solutions of linear
systems of equations per iteration or at least avoiding the exact
solution of such systems. This goal is almost realized in the case
of bound-constrained optimization problems and currently lot of
methods are available to solve large scale box-constrained
optimizations. In this context we can recall the active set Newton
algorithm of \cite{facchinei2002truncated}, affine-scaling
interior-point Newton methods of
\cite{heinkenschloss1999superlinear} and new active set algorithm of
\cite{hager2006new}. One of the most important properties of these
methods is relaxing the strict complementarity condition to the
strong second order sufficient optimality condition to prove the
local superlinear rate of convergence.

It seems that, a key to develop an efficient method to solve large
scale knapsack constrained optimization problems is to adapt large
scale box-constrained solvers for such structures; while preserving
their efficiency and convergence theories. There were a few works in
which this clue was taken into account. In \cite{dai2006new}, by
introducing an efficient method to project a trial step on to the
knapsack constraints, the projected Barzilai-Borwein method
\cite{dai2005projected} was extended to solve knapsack constrained
optimization problems. Modifying the search direction to keep
iterations interior with respect to the related linear constraint,
\cite{gonzalez2009aas} extended the affine-scaling interior-point
CBB method \cite{hager2009affine} to solve knapsack constrained
optimization problems. However, the mentioned methods possess a
linear local rate of local convergence at the best conditions.

Roughly speaking, the goal of this paper is to introduce a general
framework to extend almost every bound-constrained optimization
method to solve knapsack constrained problems without destroying its
efficiency and convergence theories. For this purpose we shall
specifically concentrate on new active set algorithm by
\cite{hager2006new}. The main reasons for this choice are its
excellent convergence and efficiency in contrast to alternative
methods. Moreover, we briefly discuses the extension of the
Bigin-Marinez active-set algorithm (GENCAN solver)
\cite{birgin2002large} to solve knapsack problems. This algorithm is
very similar to that of Hager and Zhang. It is worth mentioning that
extension of this algorithm to solve linearly constrained nonlinear
programs is already performed in \cite{andretta2010partial}, which
is called as GENLIN solver. However, it is not a computationally
feasible algorithm for large scale problems. We are hopeful that,
following the presented procedure, it be possible to extend other
alternative methods for such problems without important technical
difficulties.

The efficient and stable implementation of the projection onto the
knapsack constrains and null-space treatment of the linear
constraint are main ingredients of the presented framework for
mentioned extension in this study. Our method to project a point
onto the space of knapsack constrains is partly identical to that of
\cite{dai2006new,kiwiel2008breakpoint}. But, the technical details
are different and some new results are also included. According to
our knowledge, an efficient and stable implementation of the
null-space method for knapsack constrains is introduced in this
study for the first time.

{\bf Notations. } For any scalar $v \in \mathbb{R}$, $v^+ = \max\{v,
0\}$ and $v^- = \min\{v, 0\}$. The median operator acting on triple
$\{u, v, w\}$ is denoted by $\mathtt{mid}(u, v, w)$, i.e.,
$\mathtt{mid}(u, v, w) = \max\{u, \min\{v, w\} \}$. For an arbitrary
vector $\mathbf{v} \in \mathbb{R}^n$ and bound vectors $\mathbf{l},
\mathbf{u} \in \mathbb{R}^n$ $(\mathbf{l} \leqslant \mathbf{u})$,
the operator $\mathtt{mid}(\mathbf{l}, \mathbf{v}, \mathbf{u})$
results a vector $\mathbf{w} \in \mathbb{R}^n$ such that $w_i =
\mathtt{mid}(l_i, v_i, u_i), \ i=1, \ldots, n$. The gradient and
Hessian of the objective function $f(\mathbf{x})$ with respect to
$\mathbf{x}$ are denoted by $\nabla_\mathbf{x} f(\mathbf{x})$ and
$\nabla^2_\mathbf{x} f(\mathbf{x})$ respectively. For any
$\mathbf{x} \in \mathcal{B}$ the active and inactive index-set (with
respect to bound constraints) are denoted by
$\mathcal{A}(\mathbf{x})$ and $\mathcal{I}(\mathbf{x})$
respectively, i.e.%
\[%
\mathcal{A}(\mathbf{x}) = \{i \in [0,n] : \ x_i = l_i \ \mathtt{or}
\ x_i = u_i\}
\]%
\[
\mathcal{I}(\mathbf{x}) = \{i \in [0,n] : \ l_i < x_i < u_i\}\
 \qquad
\]%
The operator $\mathtt{dim}\ (\mathcal{A}(\mathbf{x}))$ returns the
dimension of the active set with respect to bound constraints at
$\mathbf{x}$. Assuming $\mathtt{dim}\ (\mathcal{A}(\mathbf{x}))=t$,
the operator $\mathtt{shrink(\mathbf{x})}$ gives the row (column)
vector $\mathbf{x} \in \mathbb{R}^n$ as the input argument and
returns the reduced row (column) vector $\mathbf{v} \in
\mathbb{R}^{n-t}$ which includes only entries of $\mathbf{x}$
corresponding to the inactive indices. Similarly, the operator
$\mathtt{expand(\mathbf{v})}$ gives the row (column) vector
$\mathbf{v} \in \mathbb{R}^{n-t}$ as the input argument and returns
the expanded row (column) vector $\mathbf{x} \in \mathbb{R}^n$ by
adding the active indices to $\mathbf{v}$ accordingly. The subscript
$k$ ($\square_k$) is often used to denote the quantity of interest
at the $k$-th iteration. When there is no confusion, we may use
$\mathbf{v}_k$ to denote $\mathbf{v}(\mathbf{x}_k)$.


\section{Projection onto the knapsack constraints}
\label{sec:projection}

As it was mentioned in section \ref{sec:int}, the projection of a
trial point onto the feasible set $(\mathcal{D})$ is one of the main
ingredient of the presented method in this study. This issue is
studied in this section. Consider the trial point $\mathbf{y} \in
\mathbb{R}^n$, the projection of $\mathbf{y}$ onto the feasible set;
which is denoted by $\mathbf{z} \in \mathbb{R}^n$ here; is equal to
solution of the following constrained
least square problem%
\begin{equation}%
\label{eq:proj}%
    \mathbf{z} := \mathcal{P}_\mathcal{D}(\mathbf{y}) =
                  \arg \min_{\mathbf{x} \in \mathcal{D}} \
                  \frac 12 {\| \mathbf{x} - \mathbf{y} \|}_2^2%
\end{equation}
$\mathcal{P}_\mathcal{D}(\mathbf{y})$ in \eqref{eq:proj} denotes the
projection operator which projects a trial point $\mathbf{y}$ onto
the feasible set $\mathcal{D}$ with respect to the euclidean norm.
In the first part of this section we shall consider projection onto
 $\mathcal{D}_\mathcal{E}$. Later we extend our method to solve
 problem of projection onto  $\mathcal{D}_\mathcal{I}$. Solution of \eqref{eq:proj}
 with $\mathcal{D}=\mathcal{D}_\mathcal{E}$ is equivalent to solving
 the following quadratic programming (QP) problem%
\begin{equation}%
\label{eq:proj_e}%
    \min \quad \frac 12\ \mathbf{x}^T \mathbf{I} \ \mathbf{x} -
    \mathbf{y}^T \mathbf{x} %
    \quad \mathtt{s.t.:} \quad \mathbf{a}^T \mathbf{x} = b,%
    \quad \mathbf{l} \leqslant \mathbf{x} \leqslant \mathbf{u},%
\end{equation}
where $\mathbf{I} \in \mathbb{R}^{n\times n}$ denotes the identity
matrix. In general solution of a QP problem is expensive, however,
exploiting specific structure of \eqref{eq:proj_e} results an
efficient computational method which is discussed here.

\begin{proposition}%
\label{propos:fs}%
The feasible set $\mathcal{D}_\mathcal{E}$ is non-empty if the
following conditions hold%
\[%
\sum_{i=1}^{n} ( u_i a_i^- + l_i a_i^+)%
\leqslant b \leqslant %
\sum_{i=1}^{n} ( u_i a_i^+ + l_i a_i^- )%
\]%
\end{proposition}%

\begin{proof}
The proof is trivial considering the geometry of
$\mathcal{D}_\mathcal{E}$ (also see equation 2.6 in
\cite{dai2006new}).
\end{proof}

\begin{theorem}%
\label{th:proj_e}%
Assume that the feasible domain $\mathcal{D}_\mathcal{E}$ is
non-empty. Then problem \eqref{eq:proj_e} has a unique solution
$\mathbf{x}^*$
which is computed by %
\begin{equation}%
\label{eq:solution_proj_e}%
\mathbf{x}^* = \mathtt{mid} (\mathbf{l}, \ \mathbf{y} - \lambda^*
\mathbf{a}, \ \mathbf{u})%
\end{equation}%
where $\lambda^* \in \mathbb{R}^n$ is the Lagrange multiplier
corresponding to linear equality constraint $\mathbf{a}^T \mathbf{x}
= b$ at the unique KKT point of
\eqref{eq:proj_e}, and  $\lambda^*$ is equal to the unique root of the following equation%
\[
h(\lambda) = b - \sum_{i=1}^n \big[\ a_i\ \mathtt{mid} (l_i, \ y_i -
\lambda a_i, \ u_i)\ \big].
\]
\end{theorem}%

\begin{proof}%
The existence and uniqueness of solution and Lagrange multiplier is
 followed directly by the strict convexity of the problem and
non-emptiness assumption of the feasible domain. Now let us to
reformulate  \eqref{eq:proj_e} as a box-constrained optimization
problem by augmenting the linear equality constraint to the objective function%
\begin{equation}%
\label{eq:al_th_proj_e}%
\mathcal{L}[\textbf{x}; \lambda]= \frac 12\ \mathbf{x}^T \mathbf{I}
\ \mathbf{x} - \mathbf{y}^T \mathbf{x} + \lambda (\textbf{a}^T
\textbf{x} - b)%
\quad %
\mathtt{s.t.:} %
\quad \mathbf{l} \leqslant \mathbf{x} \leqslant \mathbf{u}%
\end{equation}%
where $\lambda \in \mathbb{R}$ is the lagrange multiplier
corresponding to the linear equality constraint. It is obvious that
the unique stationary point of Lagrangian $\mathcal{L}[\textbf{x};
\lambda]$ constrained by box $\mathcal{B}$ which is shown by
$(\textbf{x}^*, \lambda^*)$ is equal to the solution of
\eqref{eq:proj_e}.

For a fixed value of $\lambda$, the box constrained stationary point
of \eqref{eq:al_th_proj_e}, $\textbf{x}^*(\lambda)$, can be computed
by the projected gradient method (cf. \cite{lin1999newton}). Using
optimality conditions based on the projected gradient method,
$\textbf{x}^*(\lambda)$ is equal to the unique zero of projected
gradient (with respect to \textbf{x})
of Lagrangian $\mathcal{L}$, i.e.,%
\begin{equation}%
\label{eq:th_proj_e_x_lambda}%
    \mathcal{P}_{\mathcal{B}}
    \big(\textbf{x}^*(\lambda) -
    \nabla_{\textbf{x}}\mathcal{L}[\textbf{x}^*(\lambda);
    \lambda]\big) - \textbf{x}^*(\lambda)= 0
\end{equation}%
where $\mathcal{P}_{\mathcal{B}}(\cdot)$ denotes the projection
operator onto $\mathcal{B}$. Simplification of
\eqref{eq:th_proj_e_x_lambda} results%
\begin{equation}%
\label{eq:th_proj_e_x_lambda2}%
    \mathcal{P}_{\mathcal{B}}
    \big(\textbf{y} - \lambda \textbf{a}\big) - \textbf{x}^*(\lambda)= 0
\end{equation}%
since for an arbitrary vector $\mathbf{v} \in \mathbb{R}^n$,
$\mathcal{P}_{\mathcal{B}}(\mathbf{v})$ is equal to
$\mathtt{mid}(\mathbf{l}, \mathbf{v}, \mathbf{u})$, the solution of
equation \eqref{eq:th_proj_e_x_lambda2} can be written in the
following
explicit form%
\begin{equation}%
\label{eq:x_i_lambda}%
\mathbf{x}^*(\lambda) = \mathtt{mid} (\mathbf{l}, \ \mathbf{y} -
\lambda \mathbf{a}, \ \mathbf{u})
\end{equation}%
considering the necessary optimality conditions of
\eqref{eq:proj_e}, $\lambda^*$ is equal to unique zero of the
following non-smooth equation%
\begin{equation}%
\label{eq:e_lambda}%
h(\lambda) = b - \sum_{i=1}^n a_i x_i^*(\lambda)
\end{equation}%
which completes the proof.
\end{proof}

Therefore, finding the unique solution of single parameter
non-smooth equation $h(\lambda) = 0$ is the main step to solve
\eqref{eq:proj_e}. Function $h(\lambda)$ is a piecewise linear
function with at most $2n$ breakpoints. The corresponding $\lambda$
of these breakpoints can be shown by set%
\[ \mathcal{T}_\lambda = \{\ \lambda_i^l, \ \lambda_i^u \ | \  i=1,
\ldots, n; \ a_i \neq 0 \}\]%
where%
\[
\lambda_i^l = (y_i- l_i)/a_i,  \quad \lambda_i^u = (y_i- u_i)/a_i,
\quad i=1, \ldots, n,
\]%
it is evident that $\lambda_i^u \leqslant \lambda_i^l$ for $a_i>0$
and $\lambda_i^l \leqslant \lambda_i^u$ for $a_i<0$. Considering
$\lambda$ coordinates of breakpoints, for $a_i> 0$, $x_i(\lambda)$
can be expressed in the following form, %
\begin{equation}%
\label{eq:x_i_lambdap}%
    x_i(\lambda)=%
    \left\{%
\begin{array}{lll}
      u_i, & \mathtt{if} & \lambda \leqslant \lambda_i^u,\\
      y_i - \lambda a_i,
      & \mathtt{if} &  \lambda_i^u \leqslant \lambda  \leqslant \lambda_i^l,\\
      l_i,& \mathtt{if} & \lambda \geqslant \lambda_i^l.
\end{array}%
\right.%
\end{equation}%
in the same way for $a_i< 0$, we have
\begin{equation}%
\label{eq:x_i_lambdam}%
    x_i(\lambda)=%
    \left\{%
\begin{array}{lll}
      u_i, & \mathtt{if} & \lambda \geqslant \lambda_i^u,\\
      y_i - \lambda a_i,
      & \mathtt{if} &  \lambda_i^l \leqslant \lambda  \leqslant \lambda_i^u,\\
      l_i,& \mathtt{if} & \lambda \leqslant \lambda_i^l.
\end{array}%
\right.%
\end{equation}%
note that for $a_i=0$, $x_i$ is independent from $\lambda$ and is
equal to $\mathtt{mid}(l_i, \ y_i, \ u_i)$. In this case, it is
possible to consider a reduced counterpart of the original problem.
Therefore, without loss of generality, we can consider $a_i \neq 0$.
Considering \eqref{eq:x_i_lambdap} and \eqref{eq:x_i_lambdap},
$x_i(\lambda)$ is a continuous piecewise linear and monotonically
nonincreasing function of $\lambda$ for $a_i>0$; and in the same
way; $x_i(\lambda)$ is a continuous piecewise linear and
monotonically nondecreasing function of $\lambda$ for $a_i<0$.
Therefore, according to \eqref{eq:x_i_lambda} and
\eqref{eq:e_lambda},  $h(\lambda)$ is a continuous piecewise linear
and monotonically nonincreasing function of $\lambda$. It is obvious
that finding the root of \eqref{eq:e_lambda} is the most expensive
part of projection in this section.

In \cite{kiwiel2008breakpoint} a median based breakpoint searching
algorithm was introduced to find the unique root of
\eqref{eq:e_lambda}. This algorithm is somehow bisecting the set of
breakpoints (consider the traditional interval bisection method).
Although, by using O(n) median finding method, the computational
complexity of this algorithm is O(n), the proportionality constant
 and the amount of operations per step is relatively large. In
\cite{dai2006new} a two-phase algorithm was suggested to find the
root of \eqref{eq:e_lambda}. In the first stage, the bracketing
phase, an interval $[\lambda_L, \lambda_R]$ where $h(\lambda_L)\cdot
h(\lambda_R)<0$ is found. Then in the second phase, secant step, the
root of \eqref{eq:e_lambda} is computed by an accelerated secant
algorithm. In the present study, we use a combination of the
bisection and quadratic interpolation methods to find the root of
\eqref{eq:e_lambda}.

\begin{proposition}%
\label{propos:root_interval}
Assume %
$\lambda_L = \min \{\ \lambda_i^l, \ \lambda_i^u \ | \ i=1, \ldots,
n; \ a_i \neq 0 \}$ %
and %
$\lambda_R = \max \{\ \lambda_i^l, \ \lambda_i^u \ | \ i=1, \ldots,
n; \ a_i \neq 0 \}$ %
then $h(\lambda_L)\cdot h(\lambda_R) < 0$ or $h(\lambda_L)\cdot
h(\lambda_R) = 0$, i.e., the root of \eqref{eq:e_lambda} happens
within the interval $(\lambda_L, \lambda_R)$ or it is equal to
either of $\lambda_L$ and $\lambda_R$. Moreover, $h(\lambda)
\leqslant 0$ for $\lambda \leqslant \lambda_L$ and $h(\lambda)
\geqslant 0$ for $\lambda \geqslant \lambda_R$.
\end{proposition}%

\begin{proof}%
Since $h(\lambda)$ is a continuous piecewise linear and
monotonically nonincreasing function of $\lambda$ and it has a
unique root, we should have $h(\lambda) \leqslant 0$ for $\lambda
\leqslant \lambda_L$ and $h(\lambda) \geqslant 0$ for $\lambda
\geqslant \lambda_R$. So the root of $h(\lambda)$ happens within the
interval $[\lambda_L, \lambda_R]$.
\end{proof}%

Considering the Proposition \ref{propos:root_interval}, the
bracketing phase used in \cite{dai2006new} is not required as it is
a-priori known. Therefore, it is possible to find the root of
$h(\lambda)$ by the traditional interval bisection method.

Assume that the error at the $k$-th step of the bisection algorithm
is shown by $\epsilon_k$, where $\epsilon_0 =
(\lambda_R-\lambda_L)/2$. Then $\epsilon_{k+1} = \epsilon_k/2$.
Therefore the number of bisection steps to find $\lambda^*$ within
tolerance $\epsilon$ is equal to $\log_2 (\epsilon_0/\epsilon)$.
Considering the limited arithmetic precision of computers, the upper
bound on number of bisection steps is priori known and is
independent from $n$. Notice that the number of set bisection in
\cite{kiwiel2008breakpoint} is function of $n$. Therefore, for large
$n$, the worst case complexity of the interval bisection method
should be better than the worst case complexity of the set bisection
algorithm of \cite{kiwiel2008breakpoint}. Moreover, when the
interval became sufficiently small (in the interval bisection
algorithm), it is possible to do a linear interpolation between
three consecutive available points to find a trial root. Then
discarding the algorithm if the trial root is zero of $h(\lambda)$.

In contrast to alternative methods, the bisection method is the most
robust algorithm and is enable to find the root within the machine
precision in a finite number of steps. However, its linear rate of
convergence is cumbersome. In the present study we use the algorithm
introduced in  \cite{brent2002algorithms} to find zero of
$h(\lambda)$. This algorithm is a combination of the bisection and
the inverse quadratic interpolation methods. By this strategy, Brent
algorithm benefits from the good convergence-rate of quadratic
interpolation and the robustness of the bisection method. Prior to
describing this method, it is worth to state the following remark to
decrease the computational complexity (hopefully with a $\log_2$
slope) in the course of iterations. Assume that the current search
interval is shown by $[\lambda_l, \lambda_r]$.
\begin{remark}%
Considering \eqref{eq:x_i_lambdap} and \eqref{eq:x_i_lambdam}
together with the current value of $[\lambda_l, \lambda_r]$, it is
possible to reduce the computational cost by freezing the components
of $\mathbf{x}$ vector for which $x_i$ is determined. More clearly,
when $a_i>0$, $x_i^*(\lambda)$ can be fixed to $u_i$ or $l_i$ if
$\lambda_r \leqslant \lambda_i^u$ or $\lambda_l \geqslant
\lambda_i^l$ respectively; and when $a_i<0$, $x_i^*(\lambda)$ can be
fixed to $u_i$ or $l_i$ if $\lambda_l \leqslant \lambda_i^u$ or
$\lambda_r \geqslant \lambda_i^l$ respectively.
\end{remark}%

Consider $\epsilon_M \in \mathbb{R}^+$ as the machine precision, we
are looking for the unique root of $h(\lambda)$ within precision
$\epsilon \in \mathbb{R}^+$ $(\epsilon > \epsilon_M)$. The machine
precision can be computed by algorithm 665 of ACM TOMS
\cite{Cody1988a}. In the root finding algorithm in this study we
have three points $\lambda_a$, $\lambda_b$ and $\lambda_c$ such that
$h(\lambda_b)\cdot h(\lambda_c) \leqslant 0$, $|h(\lambda_b)|
\leqslant |h(\lambda_c)|$, and $\lambda_a$ may coincide with
$\lambda_c$. Initially $\lambda_b = \lambda_L$, $\lambda_c =
\lambda_R$ and $\lambda_a = \lambda_c$. The point $\lambda_b$ is
considered as the best approximation to $\lambda^*$ in the course of
iterations.

Consider $\Delta_\lambda = (\lambda_c-\lambda_b)/2$. If
$\Delta_\lambda \leqslant \epsilon$ the value of $\lambda_b$ is
returned as an approximation to $\lambda_*$, else an inverse
quadratic interpolation is used to compute a trial approximation for
the root of function $h$. For the convenience, $h(\lambda_x)$ is
shown by $h_x$ henceforth. Assume we have three distinct points
$(\lambda_a, h_a)$, $(\lambda_b, h_b)$ and $(\lambda_c, h_c)$, then
the quadratic lagrange interpolation formulae
can be written as follows%
\[
\lambda = %
\frac%
{(h_\lambda - h_b)(h_\lambda - h_c)}%
{(h_a - h_b)(h_a - h_c)}\ \lambda_a%
+%
\frac%
{(h_\lambda - h_c) (h_\lambda - h_a)}%
{(h_b - h_c)(h_b - h_a)}\ \lambda_b%
+%
\frac%
{(h_\lambda - h_a)(h_\lambda - h_b)}%
{(h_c - h_a)(h_c - h_b)}\ \lambda_c%
\]
The root of this quadratic approximation, $\lambda_t$, can be
written in the
following explicit form%
\begin{equation}%
\label{eq:lambda_trial}%
    \lambda_t = \lambda_b + p/q
\end{equation}
where
$r = h_b/h_c$, $s=h_b/h_a$, $t = h_a/h_c$, $q = \pm (t-1)(r-1)(s-1)$
and $p = \pm s t (r-t)(\lambda_c - \lambda_b) - s (1-r) (\lambda_b -
\lambda_a)$. At the start of iterations in which there are only two
distinct points, a linear interpolation is used to find the trial
root. When $q \approx 0$ the overflow problem may cause to the
failure of computation. Therefore, the trial root $\lambda_t$ will
be rejected in this case. Moreover, when the interpolating parabola
has two roots between $\lambda_b$ and $\lambda_c$ or when its root
is located outside of this interval, the interpolation is poor and
the trial root $\lambda_t$ should be rejected. In the case of
inefficient step $(|p| \leqslant \epsilon |q|)$ or when the current
$p/q$ is greater than half of its previous value, the trial root
$\lambda_t$ will be rejected to avoid the slow convergence. Anyway
the trail root which is resulted from the inverse interpolation be
rejected, a bisection step will be performed. In summary, if either
of the following conditions is met the trial root of the inverse
quadratic interpolation is rejected and a bisection step is used
instead
\[
|p| \geqslant 2/3 \ |q \Delta_\lambda|, \quad  |p| \leqslant
\epsilon |q|, \quad |p/q| \geqslant 1/2 \ |p/q|_{old}
\]

The stopping criteria is a critical issue to avoid excess
computations when rounding errors prevent further progress toward
the exact solution in the vicinity of root. Following
\cite{wilkinson1994rounding}, the following relation is used as the
stopping tolerance in this study%
\[
tol = 2 \epsilon_M |\lambda_b| + \epsilon/2
\]
Note that the above criterion does not tell us how close we are to
the root, but only that we are in some interval about the zero where
roundoff error may be dominating our calculations.

 According to our numerical experiments, this method finds
the desired root within the machine precision by a fewer function
evaluations in contrast to mentioned alternatives (in our experience
the number of function evaluation calls were usually below 12 for
double-precision arithmetic and  $\epsilon=1.e-15$).

\begin{remark}%
Similar to new line-search algorithm introduced in
\cite{hager2005new}, the mentioned root finding method robustly
tolerates the limited machine precision; simultaneously remains
efficient as much as possible. Therefore, it has a good potential to
be adapted as an alternative robust line-search method. In
particular, in contrast to that of \cite{hager2005new} which uses
inverse linear interpolations (secant steps), it uses inverse
quadratic interpolations without additional function evaluation.
\end{remark}%

Now, consider the solution of \eqref{eq:proj} with
$\mathcal{D}=\mathcal{D}_\mathcal{I}$ which is equivalent to the
following
QP problem%
\begin{equation}%
\label{eq:proj_i}%
    \min \quad \frac 12\ \mathbf{x}^T \mathbf{I} \ \mathbf{x} -
    \mathbf{y}^T \mathbf{x} %
    \quad \mathtt{s.t.:} \quad b_l \leqslant \mathbf{a}^T \mathbf{x} \leqslant b_u,%
    \quad \mathbf{l} \leqslant \mathbf{x} \leqslant \mathbf{u},%
\end{equation}

\begin{proposition}%
\label{propos:fsi}%
The feasible set $\mathcal{D}_\mathcal{I}$ is non-empty if the
following conditions hold%
\[%
\sum_{i=1}^{n} ( u_i a_i^- + l_i a_i^+)%
\leqslant b_l \leqslant b_u \leqslant %
\sum_{i=1}^{n} ( u_i a_i^+ + l_i a_i^- )%
\]%
\end{proposition}%

\begin{proof}
The proof is directly followed from proposition \ref{propos:fsi}.
\end{proof}

The following theorem provides an elegant method to compute the
solution of \eqref{eq:proj_i} using results of theorem
\ref{th:proj_e}.

\begin{theorem}%
\label{th:proj_ie}%
Assume that the feasible set $\mathcal{D}_{\mathcal{I}}$ is
non-empty. Then problem \eqref{eq:proj_i} has a unique solution
$\mathbf{x}^*$ which is computed by the following relation%
\begin{equation}%
\label{eq:solution_proj_i}%
\mathbf{x}^* = \mathtt{mid} (\mathbf{x}_L^*, \ \mathbf{y}, \ \mathbf{x}_U^*)%
\end{equation}%
where $\mathbf{x}_L^*, \mathbf{x}_U^* \in \mathbb{R}^n$ are unique
solutions of the following QP problems%
\begin{align}%
\label{eq:proj_i_l}%
    \mathbf{x}_L :=& \ \min \quad \frac 12\ \mathbf{x}^T \mathbf{I} \ \mathbf{x} -
    \mathbf{y}^T \mathbf{x} %
    \quad \mathtt{s.t.:} \quad  \mathbf{a}^T \mathbf{x} = b_l,%
    \quad \mathbf{l} \leqslant \mathbf{x} \leqslant \mathbf{u},%
\\
\label{eq:proj_i_u}%
    \mathbf{x}_U := & \ \min \quad \frac 12\ \mathbf{x}^T \mathbf{I} \ \mathbf{x} -
    \mathbf{y}^T \mathbf{x} %
    \quad \mathtt{s.t.:} \quad \mathbf{a}^T \mathbf{x} = b_u,%
    \quad \mathbf{l} \leqslant \mathbf{x} \leqslant \mathbf{u},%
\end{align}
\end{theorem}

\begin{proof}
The existence and uniqueness of solutions for problems
\eqref{eq:proj_i}, \eqref{eq:proj_i_l} and \eqref{eq:proj_i_u} are
directly followed by the strict convexity of related problems and
the non-emptiness assumption of $\mathcal{D}_{\mathcal{I}}$.

Same as previous, we reformulate \eqref{eq:proj_i} as a bound
constrained optimization problem by augmenting the linear inequality
constraints to the objective function%
\[
\mathcal{L}[\textbf{x}; \mu_l; \mu_u]= \frac 12\
\mathbf{x}^T \mathbf{I} \ \mathbf{x} - %
\mathbf{y}^T \mathbf{x} + %
\mu_l (\textbf{a}^T \textbf{x} - b_l) + %
\mu_u (\textbf{a}^T \textbf{x} - b_u) %
\quad %
\mathtt{s.t.:} %
\quad \mathbf{l} \leqslant \mathbf{x} \leqslant \mathbf{u}%
\]
where $\mu_l, \mu_u \in \mathbb{R}$ are the lagrange multipliers
corresponding to the inequality constraints $\textbf{a}^T \textbf{x}
\geqslant b_l$ and $\textbf{a}^T \textbf{x} \leqslant b_u$
respectively. It is evident that the unique stationary point of
Lagrangian $\mathcal{L}[\textbf{x}; \mu_l; \mu_u]$ constrained by
box $\mathcal{B}$ which is shown by $(\textbf{x}^*, \mu_l^*,
\mu_u^*)$ is equal to the unique solution of \eqref{eq:proj_i}.

For fixed values of $\mu_l$ and $\mu_u$, the bound constrained
stationary point of Lagrangian $\mathcal{L}$, i.e.,
$\textbf{x}^*(\mu_l, \mu_u)$, can be computed by the projected
gradient method. Using optimality conditions based on the projected
gradient method, $\textbf{x}^*(\mu_l, \mu_u)$ is equal to the unique
zero of the projected gradient (with respect to \textbf{x}) of
Lagrangian
$\mathcal{L}$. Using the same way like \eqref{eq:th_proj_e_x_lambda} results%
\begin{equation}%
\label{eq:x_lambda_in}%
\mathbf{x}^*(\mu_l, \mu_u) = \mathtt{mid} \big(\mathbf{l}, \
\mathbf{y} - (\mu_l+\mu_u) \mathbf{a}, \mathbf{u} \big)
\end{equation}%
assuming $\mu = \mu_l+\mu_u$ simplifies \eqref{eq:x_lambda_in} to
the following form
\begin{equation}%
\label{eq:x_lambda_in2}%
\mathbf{x}^*(\mu) = \mathtt{mid}(\mathbf{l}, \ \mathbf{y} - \mu
\mathbf{a}, \mathbf{u})
\end{equation}%
note that at least either of $\mu_l$ or $\mu_u$ is zero at the
optimal solution, so, as \eqref{eq:x_lambda_in2} shows, knowing
$\mu_l+\mu_u$ at the optimal solution is sufficient to uniquely
determine $\mu_l^*$ and $\mu_u^*$. Considering the necessary
optimality conditions of \eqref{eq:proj_i}, the remaining job is to
compute the optimal value of $\mu$ (which is shown by $\mu^*$) such
that
the following inequalities are satisfied%
\begin{equation}%
\label{eq:x_lambda_in3}%
b_l \leqslant \mathbf{a}^T \mathbf{x}^*(\mu) \leqslant b_u %
\end{equation}%
or
\begin{equation}%
\label{eq:x_lambda_in4}%
b_l \leqslant \mathbf{a}^T \mathtt{mid}(\mathbf{l}, \ \mathbf{y} -
\mu \mathbf{a}, \mathbf{u}) \leqslant b_u %
\end{equation}%
due to the uniqueness of lagrange multiplier $\mu^*$ at the optimal
solution, the inequality equation \eqref{eq:x_lambda_in4} has only
one solution. Without confusion with the previous definition for
$h(\lambda)$, assume the following new definition for function $h(\mu)$%
\[
h(\mu) = \mathbf{a}^T \mathbf{x}^*(\mu)
\]%
Similar to the previous arguments on function $h(\lambda)$, it is
easy to show that $h(\mu)$ is a continuous piecewise linear and
monotonically nonincreasing function of $\mu$ with at most $2n$
breakpoints (to save space we avoid repetition of similar
discussions for $h(\mu)$).

Assume $\lambda_l^*$ and $\lambda_u^*$ are optimal lagrange
multipliers corresponding to linear equality constraint in problems
\eqref{eq:proj_i_l} and \eqref{eq:proj_i_u} respectively.
Considering the mentioned monotonicity of functions $h(\lambda)$ and
$h(\mu)$, the following inequalities are identical%
\begin{equation}%
\label{eq:x_lambda_in5}%
\lambda_u^* \leqslant  \mu^* \leqslant  \lambda_l^* %
\end{equation}%
therefore an appropriate search interval for $\mu^*$ is the bracket
$[\lambda_u^*, \lambda_l^*]$. Considering
\eqref{eq:solution_proj_e}, \eqref{eq:x_lambda_in4} and
\eqref{eq:x_lambda_in5}, the following inequalities hold%
\begin{equation}%
\label{eq:x_lambda_in6}%
\mathbf{x}^*_L \leqslant  \mathbf{x}^* \leqslant  \mathbf{x}^*_U %
\end{equation}%
notice that the inequalities are understood componentwise here.
Refereing again to the mentioned monotonicity of functions
$h(\lambda)$ and $h(\mu)$, when $\mathbf{x}^*_L \leqslant
\mathbf{x}^*$ the inequality $\textbf{a}^T \textbf{x}^* \geqslant
b_l$ always holds. In the same way, the inequality $\textbf{a}^T
\textbf{x}^* \leqslant b_u$ will be remained always satisfied when
$\mathbf{x}^* \leqslant \mathbf{x}^*_U$. Therefore, replacing bound
constraints $\mathbf{l}$ and $\mathbf{u}$ respectively with
$\mathbf{x}^*_L$ and $\mathbf{x}^*_U$ reduces problem
\eqref{eq:proj_i} to the following
bound constrained QP problem%
\begin{equation}%
\label{eq:proj_i_reduced}%
    \min \quad \frac 12\ \mathbf{x}^T \mathbf{I} \ \mathbf{x} -
    \mathbf{y}^T \mathbf{x} %
    \quad \mathtt{s.t.:} \quad \mathbf{x}^*_L \leqslant \mathbf{x} \leqslant \mathbf{x}^*_U%
\end{equation}
with the explicit solution%
\[%
\mathbf{x}^* = \mathtt{mid} (\mathbf{x}_L^*, \ \mathbf{y}, \ \mathbf{x}_U^*)%
\]%
which completes the proof.
\end{proof}

Therefore, it is possible to solve \eqref{eq:proj_i} in expense of
solving two problems with structures similar to \eqref{eq:proj_e}.
However, the asymptotic computational cost is not essentially
duplicated in this way. This is because of some shared calculations
like computing $\lambda_L$, $\lambda_R$, $h(\lambda_L)$ and
$h(\lambda_R)$. Moreover, the trajectory points produced during the
solution of \eqref{eq:proj_i_u} can be easily used to reduce the
initial search bracket to solve \eqref{eq:proj_i_l}. For instance,
if pair $(\lambda_x, h_x)$ was produced during the solution of
\eqref{eq:proj_i_u}, the corresponding $h$ to $\lambda_x$ when
solving \eqref{eq:proj_i_l} is equal to $h_x + b_l - b_u$.


\section{Null-space management of linear constraints}%
\label{sec:nullspace}%

In this section we present an O(n) method (in terms of computational
cost and consumed memory) to treat the linear constraint in knapsack
problem. For the convenience, the basic idea of the null-space
methods is recalled at the first part of this section (for more
details see chapter 5 of \cite{gill1981practical}).

Let $\mathbf{A} \in \mathbb{R}^{m\times n}$ be a full rank  matrix.
The null-space of $\mathbf{A}$ is denoted by%
\[\mathcal{N}(\mathbf{A}) = \{\mathbf{p}\in \mathbb{R}^n: \ \mathbf{A}\mathbf{p} = 0\}\]
which is the set of vectors orthogonal to the rows of $\mathbf{A}$.
The null-space of $\mathbf{A}$ is a subspace of $\mathbb{R}^n$ with
dimension $n-m$ (consider the full-rank assumption of $\mathbf{A}$).
Therefore, any linear combination of two vectors in
$\mathcal{N}(\mathbf{A})$ is also in $\mathcal{N}(\mathbf{A})$. Any
matrix $\mathbf{Z} \in \mathbb{R}^{n\times (n-m)}$ whose columns
form a basis for $\mathcal{N}(\mathbf{A})$ can be considered as a
null-space matrix for $\mathbf{A}$. It is easy to show that
$\mathbf{Z}$ satisfies $\mathbf{A}^T\mathbf{Z}=0$.

The range-space of $\mathbf{A}$ is defined as a subspace of
$\mathbb{R}^n$ with dimension $m$ which is spanned by the columns of
the $\mathbf{A}$ (the set of all linear combinations of $\mathbf{A}$
columns). In particular, we are interested in the range space of
$\mathbf{A}^T$, defined by%
\[\mathcal{R}(\mathbf{A}^T) = \{\mathbf{q}\in \mathbb{R}^n: \
\mathbf{q} = \mathbf{A}\mathbf{r} \quad \mathtt{for \ some}  \quad
\mathbf{r}\in \mathbb{R}^m\}\]
It is easy to show that $\mathcal{N}(\mathbf{A})$ and
$\mathcal{R}(\mathbf{A}^T)$ are orthogonal subspaces. Therefore, it
is possible to uniquely decompose an arbitrary vector $\mathbf{x}
\in \mathbb{R}^n$ into sum of the range-space, $\mathbf{q}\in
\mathcal{R}(\mathbf{A}^T)$, and null-space, $\mathbf{p}\in
\mathcal{N}(\mathbf{A})$, components%
\[\mathbf{x}= %
\mathbf{q}+\mathbf{p} =%
\mathbf{q} + \mathbf{Z} \mathbf{v}
 \]
where $\mathbf{v} \in \mathbb{R}^{(n-m)}$. Consider linear system of
equations $\mathbf{A}^T \mathbf{x} = \mathbf{b}$ ($\mathbf{b} \in
\mathbb{R}^m$). Assuming $\mathbf{x}_0 \in \mathbb{R}^n$ as a
particular solution for $\mathbf{A}^T \mathbf{x} = \mathbf{b}$, any
other solution $\mathbf{x}$ can be parameterized as%

\begin{equation}%
\label{eq:ns_parametrization}%
\mathbf{x} = \mathbf{x}_0 + \mathbf{Z} \mathbf{v}
\end{equation}%
In fact $\mathbf{Z} \mathbf{v}$ acts as feasible directions for
matrix $\mathbf{A}$ (note that $\mathbf{A} \mathbf{Z} \mathbf{v}=
\mathbf{0}$).

Now consider the following optimization problem%
\begin{equation}%
\label{eq:leq}%
\nonumber%
\tag{LEP}
    \min_{\mathbf{x} \in \mathbb{R}^n}
    f(\mathbf{x}) \quad \mathtt{s.t.:} \quad \mathbf{A}^T \mathbf{x} = \mathbf{b}
\end{equation}%
using \eqref{eq:ns_parametrization}, it is possible to convert the
linearly constrained optimization problem \eqref{eq:leq} on
$\mathbb{R}^n$ to the following unconstrained optimization problem
on the reduced
space $\mathbb{R}^{(n-m)}$,%
\begin{equation}%
\label{eq:reduced_leq}%
\nonumber %
\tag{RLEP}
    \min_{\mathbf{v} \in \mathcal{N}(\mathbf{A})}
    f(\mathbf{x}_0 + \mathbf{Z} \mathbf{v})
\end{equation}%
the gradient and Hessian of $f$ with respect to the reduced vector
$\mathbf{v}$ at the trial point $\mathbf{v}_0$ can be easily
computed using the chain rule, i.e.,%
\begin{align}%
\nonumber%
\nabla_\mathbf{v} f(\mathbf{x}_0 + \mathbf{Z} \mathbf{v}_0) =&
\mathbf{Z}^T \nabla_\mathbf{x} f(\mathbf{x}_0 + \mathbf{Z}
\mathbf{v}_0)\\ \nonumber%
\nabla^2_\mathbf{v} f(\mathbf{x}_0 + \mathbf{Z} \mathbf{v}_0) =
&\mathbf{Z}^T \nabla^2_\mathbf{x} f(\mathbf{x}_0 + \mathbf{Z}
\mathbf{v}_0)\mathbf{Z}%
\end{align}
the necessary and sufficient conditions for the reduced problem
\eqref{eq:reduced_leq} is same as the classical unconstrained
optimization problems in which the gradient vector and the Hessian
matrix are replaced by the reduced counterparts. Therefore, any
unconstrained optimization solver can be employed to solve
\eqref{eq:reduced_leq} without any technical difficulty.

Now, consider the following inequality constrained optimization
problem%
\begin{equation}%
\label{eq:liq}%
\nonumber%
\tag{LIP}
    \min_{\mathbf{x} \in \mathbb{R}^n}
    f(\mathbf{x}) \quad \mathtt{s.t.:} \quad \mathbf{A}^T \mathbf{x} \geqslant \mathbf{b}
\end{equation}%
Using an appropriate active set strategy, it is possible to solve
\eqref{eq:liq} with the null-space method. Assume the set of active
constraints at the local solution is known, $\hat{\mathbf{A}}$, then
the corresponding unconstrained optimization problem is solved on a
null-space spanned by only active constraints,
$\mathcal{N}(\hat{\mathbf{A}})$. In this case the caution should be
taken for inactive constraints. Assume the index set of inactive
constraints at $\mathbf{x}$ is shown by $\mathcal{I}$ (row-wise
index). Denoting the search direction at $\mathbf{x}$ by $\mathbf{p}
= \mathbf{Z} \mathbf{v}$, for all index $i \in \mathcal{I}$ so that
$\mathbf{a}_i^T \mathbf{p} \geqslant 0$, any positive move along
$\mathbf{p}$ will not violate the corresponding constraint
($\mathbf{a}_i$ denotes the $i$-th row of $\mathbf{A}$). Therefore,
constraints with non-negative $\mathbf{a}_i^T\mathbf{p}$ do not pose
any restriction on the stepsize. However, there is a critical step
length, $\gamma_i$, for indices with $\mathbf{a}_i^T \mathbf{p} <
0$, where the constraint becomes binding, i.e.,
$\mathbf{a}_i^T(\mathbf{x}+\gamma_i\mathbf{p})= \mathbf{b}_i$. So,
the upper bound on the stepsize due to feasibility of iterates is
computed by the following relation%
\begin{equation}%
\label{eq:step_inactive_c}%
    \bar{\alpha} = \min \big\{+\infty, \
    \gamma_i =
(\mathbf{b}_i - \mathbf{a}_i^T\mathbf{x})/(\mathbf{a}_i^T
\mathbf{p}) \ | \ \mathbf{a}_i^T \mathbf{p} < 0
    \big\}
\end{equation}
Similarly, the constrained optimization problem \eqref{eq:liq} can
be converted to the following unconstrained optimization problem on
the reduced space spanned by the active constraints%
\begin{equation}%
\label{eq:reduced_liq}%
\nonumber %
\tag{RLIP}
    \min_{\mathbf{v} \in \mathcal{N}(\hat{\mathbf{A}})}
    f(\mathbf{x}_0 + \alpha \mathbf{Z} \mathbf{v}), \quad \alpha \in [0, \bar{\alpha}]
\end{equation}%
By some simple linear algebra, it is easy to show that
\[
\mathtt{cond}\big(\mathbf{Z}^T\nabla^2_\mathbf{x}
f(\mathbf{x})\mathbf{Z}\big) \leqslant
\mathtt{cond}\big(\nabla^2_\mathbf{x}f(\mathbf{x})\big) \
\mathtt{cond}\big(\mathbf{Z}\big)^2%
\]%
therefore, to cope the possible instability during the computations
(consider the limited precision arithmetic of computers), it is
preferable to use an orthogonal null-space basis;
$\mathtt{cond}\big(\mathbf{Z}\big)=1$. The $QR$ factorization
\cite{golub1996matrix} is a common way to compute an orthogonal
null-space for a desired full-rank matrix.

\subsection{QR factorization}%
\label{sec:qr_factor}%

The $QR$ factorization for the transpose of full-rank matrix
$\mathbf{A} \in \mathbb{R}^{m\times n}$ is expressed in the
following form%
\[
\mathbf{A}^T = \mathbf{Q} \mathbf{R} = %
\left(%
\begin{array}{cc}
\mathbf{Q_1} & \mathbf{Q_2}
\end{array}%
\right)%
\left(%
\begin{array}{c}
\mathbf{R_1} \\
\mathbf{0}
\end{array}%
\right)%
\]
where $\mathbf{Q}$ is an orthogonal matrix, $\mathbf{Q_1} \in
\mathbb{R}^{n\times m}$, $\mathbf{Q_2} \in \mathbb{R}^{n\times
(n-m)}$, $\mathbf{R_1} \in \mathbb{R}^{m\times m}$. Since
$\mathbf{Q}$ is orthogonal, it follows that $\mathbf{A}\mathbf{Q}=
\mathbf{R}^T$, or $\mathbf{A}\mathbf{Q_1}= \mathbf{R_1}^T$ and
$\mathbf{A}\mathbf{Q_2}= \mathbf{0}$ which results $\mathbf{Z}=
\mathbf{Q_2}$.

In the present study, the Householder $QR$ algorithm (cf.
\cite{golub1996matrix}) is used to compute the $QR$ factorization of
$\mathbf{A}$. In this method the orthogonal matrix $\mathbf{Q}$ is
represented by the product of Householder reflection matrixes%
\[
\mathbf{Q} = \mathbf{H_1}\mathbf{H_2} \ldots \mathbf{H_m}
\]
where every $\mathbf{H_i} \in \mathbb{R}^{n\times n}$ is an
orthogonal matrix with the following generic form%
\[
\mathbf{H} = \mathbf{I} - \tau \mathbf{u}\mathbf{u}^T
\]
where $\mathbf{I} \in \mathbb{R}^{n\times n}$ is the identity
matrix, $\tau \in \mathbb{R}$ and $\mathbf{u} \in \mathbb{R}^n$. For
details of computing $\mathbf{H_i}$ factors refer to
\cite{golub1996matrix}. In the following we adapt the Householder
$QR$ factorization algorithm to compute an orthogonal null-space for
linear constraint in knapsack problems.

In the knapsack problems, we have at most one active linear
constraint, so we need to compute the $QR$ factorization for
$\mathbf{A} = \mathbf{a}$. Without loss of generality, assume that
$a_1 \neq 0$ (else a simple pivoting should be applied). In this
case the factor $\mathbf{Q}$ can be represented
in the following form%
\[
\mathbf{Q} = \mathbf{I} - \tau \mathbf{u}\mathbf{u}^T
\]
where %
\[
u_1 = 1, \quad u_i = \frac{a_i}{a_1-\zeta}, \quad i=2, \ldots, n,
\quad \tau = \frac{\zeta-a_1}{\zeta}, \quad %
\zeta = - \mathtt{sign}(a_1)\ \|\mathbf{a}\|_2
\]
moreover, the matrix $\mathbf{R}=\mathbf{R_1}$ is a $1\times 1$
matrix such that its singleton entry is equal to $\zeta$.


In the null-space method we frequently need to do product of
null-space matrix, $\mathbf{Z}$, or its inverse, $\mathbf{Z}^T$
(consider orthogonality of $\mathbf{Z}$), with a desired vector.
Notice that matrix $\mathbf{Z}$ is formed by removing the first
column of matrix $\mathbf{Q}$. Exploiting the specific
representation of $\mathbf{Q}$ it is possible to perform the product
of $\mathbf{Z}$ with an arbitrary vector $\mathbf{v} \in
\mathbb{R}^{n-1}$ in O(n) arithmetic operations. Considering a
pseudo entry $v_0=0$ for vector $\mathbf{v}$, simple linear algebra
results%
\[
\mathbf{zv} = \mathbf{Z}\mathbf{v}, \quad%
 {zv}_i = v_{i-1} - \tau
u_i \sum_{j=1}^{n-1}(u_{j+1}v_j),%
\quad i=1, \ldots, n.
\]
where $\mathbf{zv} \in \mathbb{R}^{n}$ is expansion of the reduced
vector $\mathbf{v} \in \mathbb{R}^{n-1}$ to the full-space.

%
In the similar way, product of $\mathbf{Z}^T$ with an arbitrary
vector $\mathbf{w} \in \mathbb{R}^n$ can be computed by the
following relation%
\[
\mathbf{ztw} = \mathbf{Z}^T\mathbf{w}, \quad%
 {ztw}_i = w_{i+1} - \tau
u_{i+1} \sum_{j=1}^{n}(u_i w_i),%
\quad i=1, \ldots, n-1.
\]
%
%

\subsection{An alternative null-space by orthogonal projection matrix}%
\label{sec:orthoproj_ns}%

The orthogonal projection matrix is another possible choice for the
null-space of full rank matrix $\mathbf{A} \in \mathbb{R}^{m\times
n}$. The orthogonal projection matrix, $\mathbf{P} \in
\mathbb{R}^{m\times n}$, can be computed by the
following relation%
\[
\mathbf{P} = \mathbf{I} - \mathbf{A}^T(\mathbf{A}\mathbf{A}^T)^{-1}
\mathbf{A}
\]
The main difference of this null-space basis with the previous
mentioned one is that the orthogonal projection matrix is not full
rank. Therefore, the equivalent unconstrained optimization problem
will be solved on the full-space $\mathbb{R}^n$. This may simplifies
the implementation of method; but probably; in expense of more
computational cost.

Notice that the name "orthogonal projection matrix" should not make
misleading about the properties of this null-space basis, as in
general $\mathbf{P}$ is not an orthogonal matrix. Therefore, the
stability of computation can be in question in using $\mathbf{P}$ as
the null-space matrix.

Without regard to the stability of computation, the special
structure of $\mathbf{A}$ in the knapsack problems, makes the
orthogonal projection matrix an attractive choice to make a
null-space basis. In this case $\mathbf{P}$ can be computed and
stored explicitly; without any cost; as follows%
\[
\mathbf{P} = \mathbf{I} - {(\mathbf{a}^T \mathbf{a})}^{-1} \
\mathbf{a} \mathbf{a}^T
\]
The product of $\mathbf{P}$ with an arbitrary vector $\mathbf{v} \in
\mathbb{R}^n$ can be computed within O(n) operations;
as follows%
\[
\mathbf{pv} = \mathbf{P} \mathbf{v}, %
\quad {pv}_i = v_i - a_i \ \frac{\sum_{j=1}^{n} (a_i v_i)}
{\sum_{j=1}^{n} (a_i a_i)},%
\quad i=1, \ldots, n
\]
where $\mathbf{pv} \in \mathbb{R}^n$ is the projection of vector
$\mathbf{v}$ onto the null-space of $\mathbf{A}$. Using theorem
\ref{th:proj_e}, it is trivial to show that the orthogonal
projection of an arbitrary point $\mathbf{y} \in \mathbb{R}^n$ onto
equality constraint of knapsack problem (ignoring bound constraints)
which is
denoted by $\mathbf{z} \in \mathbb{R}^n$ can be computed by%
\begin{equation}%
\label{eq:proj_knapsack_eo}%
z_i = y_i - a_i \ \frac{b-\sum_{j=1}^n (a_j y_j)}{\sum_{j=1}^n
a_j^2}, \quad i=1, \ldots, n%
\end{equation}%
In the similar way, using theorem \ref{th:proj_ie}, projection point
onto knapsack bilateral inequality constraint (ignoring bound
constraints) is computed by the following relation%
\begin{equation}%
\label{eq:proj_knapsack_io}%
\mathbf{z} = \mathtt{mid}(\mathbf{zl},\  \mathbf{y},\ \mathbf{zu})%
\end{equation}%
where $\mathbf{zl}$ and $\mathbf{zu}$ are computed by
\eqref{eq:proj_knapsack_eo} in which $b$ is replaced by $b_l$ and
$b_u$ respectively. Equations \eqref{eq:proj_knapsack_eo} and
\eqref{eq:proj_knapsack_io} are useful for treatment of linear
constraints by projection in the unconstrained solver.

\subsection{Unconstrained optimization on the reduced space}%
\label{sec:exploiting_reduced_solver}%

In this subsection we want to comment on the exploiting an
unconstrained optimization solver on the null-space of (active)
linear constraints. Basically this procedure is straightforward and
the unconstrained optimization solver can be blind about the
constrained problem, i.e., we can use the unconstrained optimization
solver as a black-box. The following items are sufficient to exploit
the
unconstrained solver for this purpose:%
\begin{itemize}%
\item[(a)]%
At the initial step we need the starting point $\mathbf{x}_0 \in
\mathbb{R}^n$ which is feasible with respect to linear constraints.
It can be achieved by projecting a trail point onto the constraint
set. This point is stored in memory and used during the
unconstrained optimization as described below. The starting point of
unconstrained optimization solver, $\mathbf{v}_0 \in
\mathbb{R}^{n-m}$, can be taken equal to $\mathbf{0} \in
\mathbb{R}^{n-m}$.

\item[(b)]%
When the unconstrained solver asks for the value of the objective
function at the trial point $\mathbf{v}_k \in \mathbb{R}^{n-m}$, we
should first compute the corresponding value of $\mathbf{v}_k$ on
the full-space, $\mathbf{x}_k \in \mathbb{R}^{n}$, by this relation
$\mathbf{x}_k = \mathbf{x}_0 + \mathbf{Z}\mathbf{v}_k$. Then
$f(\mathbf{x}_k)$ is passed to the unconstrained solver as the value
of the objective function at $\mathbf{v}_k$.

\item[(c)]%
Similar to item (b), when the unconstrained solver asks for the
gradient of the objective function,
$\nabla_{\mathbf{v}}f(\mathbf{v}_k)$, at the trial point
$\mathbf{v}_k$, we first compute $\mathbf{x}_k = \mathbf{x}_0 +
\mathbf{Z}\mathbf{v}_k$. Then pass $\mathbf{Z}^T
\nabla_{\mathbf{x}}f(\mathbf{x}_k)$ to the unconstrained solver as
the desired gradient.

\item[(c)]%
Similarly, when the unconstrained solver asks for the Hessian matrix
of the objective function, $\nabla^2_{\mathbf{v}}f(\mathbf{v}_k)$,
at the trial point $\mathbf{v}_k$, the value of $\mathbf{Z}^T
\nabla^2_{\mathbf{x}}f(\mathbf{x}_0 +
\mathbf{Z}\mathbf{v}_k)\mathbf{Z}$ is passed to it as the desired
Hessian matrix.

\item[(d)]%
The stepsize related to the globalization strategy of the
unconstrained solver should be restricted to interval $[0,
\bar{\alpha}]$, where $\bar{\alpha}$ is computed by
\eqref{eq:step_inactive_c}. In the present study, this item only
applied in the case of inequality knapsack problems. If the active
set of linear constraints at the optimal solution be determined in a
finite number of iterations, this stepsize restriction does not
destroy the local convergence rate of the unconstrained solver. For
the finite-cycle determination of active set, it is required that
the new active set be a sub-set of the previous one; which is
possible to ensure in the case of no non-degeneracy of the optimal
solution with respect to the linear active constraints (the lagrange
multiplier corresponding to the active constraint be non-zero at the
local solution).

\end{itemize}%
%


\section{Hager-Zhang active set algorithm for knapsack problems}%
\label{sec:HZASA}%

Consider a bound constrained nonlinear program, if the set of active
constraints at a local solution be a-priori known, it is possible to
fix these constraints to the corresponding bound values and solve an
unconstrained optimization problem on the reduced space spanned by
the free variables. This is the key idea of the active set strategy.
In practice, the set of active indices in not a-priori known,
therefore they should be identified by an appropriate prediction
correction strategy. Under appropriate conditions, this procedure
can be performed in a finite number of iterations. However, in
general it is possible to add (remove) only one index to (from) the
current active set. This increases the necessary number of
iterations in particular for large scale problems. Fortunately, it
is possible to add many constraints to working set by means of the
projected gradient method. This idea was used in
\cite{lin1999newton,heinkenschloss1999superlinear,birgin2002large,hager2006new}
to efficiently exploit the active set strategy.

In this section we adapt the Hager-Zhang active set algorithm
(HZ-ASA) \cite{hager2006new} to solve knapsack problems. Using the
previously constructed tools, this extension is trivial and the
convergence theory holds almost under the same conditions. The main
reason for selection of this algorithm is its excellent convergence
theories in addition to the promising numerical results reported in
\cite{hager2006new}. Moreover, unlike
\cite{lin1999newton,heinkenschloss1999superlinear}, this method
admits the superlinear convergence under the same conditions while
it does not need any explict form of second order information and/or
solution of system of linear equations per cycle. These properties
makes it an ideal choice to solve very large scale problems. The
current author believes that the key efficiency of this method can
be connected to its two-phase nature. Unlike the related methods
which use the same interactions in the course of optimization, this
method start with a cheap constrained first-order method and after
sufficient progress toward a local solution, branches to a (more
expensive) higher-order unconstrained solver. In other words, it
does not waste the energy at early stages by performing expensive
and accurate steps. Under certain conditions, the method only
branches once between two phases. This strategy is particularly
effective when the initial guess is poor. Note that the basic idea
for this strategy is not new and already used in
\cite{birgin2002large}.

In \cite{hager2006new} the nonmonotone spectral projected gradient
\cite{birgin2000nonmonotone} (SPG) was used to identify the working
set. In this method, the search direction is parallel to the
steepest descent direction (is descent) and the stepsize is computed
by projecting the (spectral) Barzilai-Borwein \cite{barzilai1988two}
step onto the feasible set. Moreover, the Grippo-Lampariello-Lucidi
nonmonotone line search algorithm \cite{grippo1986nonmonotone} is
used to ensure the convergence to a local minimum from an arbitrary
initial guess, simultaneously benefiting from the spectral property
of the Barzilai-Borwein stepsize as much as possible. The SPG method
can be efficiently applied to every convex constrained optimization
problem (under some common conditions) providing an efficient method
to project a trial point onto the feasible set. Therefore, by means
of O(n) projection algorithm introduced in section
\ref{sec:projection}, SPG can be effectively used to solve large
scale knapsack problems. However, the local convergence of SPG
method will not be better than linear in the vicinity of the local
solution.

Let $\alpha \in \mathbb{R}$, the scaled projected gradient,
$\mathbf{d}^\alpha(\mathbf{x})$ is defined as follows%
\[
\mathbf{d}^\alpha(\mathbf{x}) = \mathcal{P}_\mathcal{D}(\mathbf{x} -
\alpha \nabla_\mathbf{x} f(\mathbf{x})^T) - \mathbf{x}
\]
It is easy to show that $\mathbf{d}^\alpha(\mathbf{x})$ is a
constrained descent direction (e.g. see: lemma 2.1 of
\cite{birgin2000nonmonotone} or proposition 2.1 of
\cite{hager2005new}). Therefore, it can be used as the stopping
criteria to measure the distance of current iterate from the optimal
solution. Assume $\mathbf{s}_k = (\mathbf{x}_k - \mathbf{x}_{k-1})$
and $\mathbf{y}_k = (\nabla_\mathbf{x} f(\mathbf{x}_k) -
\nabla_\mathbf{x} f(\mathbf{x}_{k-1}))$. In the Barzilai-Borwein
method, the trial stepsize, $\alpha_k$, is computed from a
one-dimensional search method
based on the following relation%
\begin{equation}%
\label{eq:bb stepsize_t}%
    \alpha_{k} = \arg\min_{ \alpha \in {\mathbb R}} \ \frac 12 \ \|
    {\mathbf{D}}(\alpha) \ {\mathbf{s}}_{k-1} - {\mathbf{y}}_{k-1}\|^2 = %
    ({\mathbf{s}}_{k-1}^T {\mathbf{s}}_{k-1})/( {\mathbf{s}}_{k-1}^T
    {\mathbf{y}}_{k-1}).
\end{equation}%
where ${\mathbf{D}}(\alpha) = \frac{1}{\alpha}{\mathbf{I}}$ is a
diagonal approximation to the Hessian matrix $\nabla^2_\mathbf{x}
f(\mathbf{x})$. The spectral property of Barzilai-Borwein method can
be inferred
from the following lemma:%
\begin{lemma}%
\label{lem:spectral}%
The Barzilai-Borwein stepsize, is equal to inverse of Rayleigh
quotient, related to vector $s_{k-1}$, of the averaged Hessian of
the objective functional between two consecutive iterations $k-1$
and $k$.
\end{lemma}%
\begin{proof}%
Using the Mean-Value Theorem it is easy to show that:%
\[
   (\nabla_\mathbf{x} f(\mathbf{x}_k) - \nabla_\mathbf{x} f(\mathbf{x}_{k-1}))/(\textbf{x}_k -
   \textbf{x}_{k-1})= %
   \int_0^1 \nabla^2 f (t \textbf{x}_k + [1-t] \textbf{x}_{k-1}) dt, %
\]
therefore,
\[
   (\textbf{s}_{k-1}^T\textbf{y}_{k-1})/(\textbf{s}_{k-1}^T\textbf{s}_{k-1}) =
   \bigg(%
   \textbf{s}_{k-1}^T%
   \big(%
   \int_0^1 \nabla^2 f (\textbf{x}_{k-1} + t \textbf{s}_{k-1} ) dt
   \big)%
   \textbf{s}_{k-1}%
   \bigg)/(\textbf{s}_{k-1}^T\textbf{s}_{k-1}),%
\]%
which complete the proof.
\end{proof}%
By lemma \ref{lem:spectral}, $\Lambda_{max}^{-1} \leqslant \alpha_k
\leqslant \Lambda_{min}^{-1}$, where $\Lambda_{max}$ and
$\Lambda_{min}$ are the maximum and minimum eigenvalues of the
averaged Hessian matrix respectively. Therefore, $\alpha_k
\textbf{I}$ is a consistent approximation to the inverse of the
Hessian matrix. The statement of SPG algorithm is as follows:\\

\begin{center}%
\includegraphics[width=13.cm]{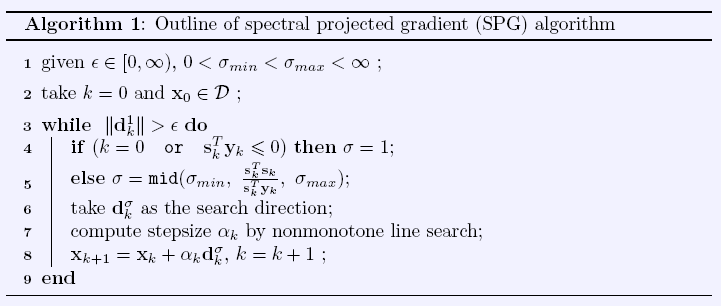}%
\end{center}%
\vspace*{10mm} %
Note that the condition $\mathbf{s}_k^T \mathbf{y}_k \leqslant 0$ in
line 4 of algorithm 1 means the detection of negative curvature
directions. Assuming that the objective function is consistent with
its quadratic approximation around the current iterate, the large
stepsize $\sigma=1$ is used to achieve the maximum reduction. The
statement of the nonmonotone line
search used in this study is as follows:\\

\begin{center}%
\includegraphics[width=13.cm]{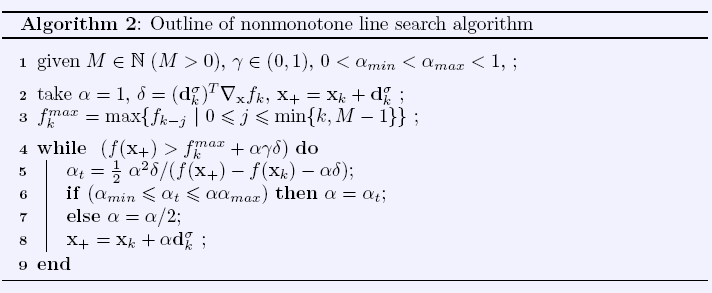}%
\end{center}%
\vspace*{10mm} %
Notice that when the decrease condition in line 4 of algorithm 2 is
not satisfied a quadratic interpolation (line 5) is used to compute
the trail step $\alpha_t$ and it is accepted if lies within interval
$[\alpha_{min}, \alpha\alpha_{max}]$, otherwise a bisection will be
performed (line 7). The quadratic function $q(\alpha)$ is formed
such that $q(0)= f(\mathbf{x}_k)$, $q(\alpha)=f(\mathbf{x}_+)$ and
$dq/d\alpha = \delta$.

\begin{theorem} (\cite{birgin2000nonmonotone} theorem 2.4) %
Algorithm SPG is well-defined, and any accumulation point of the
sequence $\{\mathbf{x}_k\}$ that it generates is a constrained
stationary point.%
\end{theorem}%

Further details about SPG algorithm and its convergence theory can
be found in either of \cite{birgin2000nonmonotone} or section 2 of 
\cite{hager2006new}.

The unconstrained solver (US) used in HZ-ASA is the CGDESCENT
algorithm \cite{hager2005new} which possesses a superlinear local
convergence rate under some common conditions. In HZ-ASA the
unconstrained solver is modified to be enable to manage infeasible
iterations (with respect to bound constraints). Similar to
\eqref{eq:step_inactive_c}, it includes the modification of the
stepsize such that the new iterates will not be infeasible; i.e.,
the trial stepsize is limited to the interval $[0, \hat{\alpha}]$
where,
\begin{equation}%
\label{eq:step_inactive_c_box}%
\hat{\alpha} = \max\{%
\alpha \in \mathbb{R}^+ \ | \ (\mathbf{x}_I(\mathbf{x}) + \alpha
\mathbf{p}_I(\mathbf{x})) \in \mathcal{B}%
\}%
\end{equation}
The outline of the reduced CGDESCENT  (RCGD)
algorithm used in this study is as follows:\\

\begin{center}%
\includegraphics[width=13.cm]{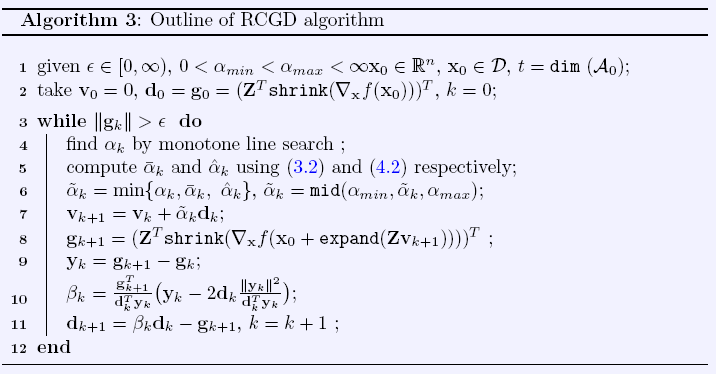}%
\end{center}%
\vspace*{10mm} %
Notice that the value of active constraints will be fixed during
RCGD iterations. The feasibility of free indices with respect to
bound constraints will be maintained by $\hat{\alpha}_k$. When
$\mathcal{D}\stackrel{def}{=}\mathcal{D}_{\mathcal{E}}$ then we
 have  always one active linear constraint and work with $\mathbf{d}$,
$\mathbf{g}$, $\mathbf{v}$, $\mathbf{y} \in \mathbb{R}^{n-t-1}$ and
$\bar{\alpha}_k = \infty$. When
$\mathcal{D}\stackrel{def}{=}\mathcal{D}_{\mathcal{I}}$ and both of
the linear inequality constraints be inactive $\mathbf{d}$,
$\mathbf{g}$, $\mathbf{v}$, $\mathbf{y} \in \mathbb{R}^{n-t}$ and
$\mathbf{Z}=\mathbf{I}$. In this case the feasibility of iterations
with respect to equality constraints will be controlled by
$\bar{\alpha}_k$. Finally, when
$\mathcal{D}\stackrel{def}{=}\mathcal{D}_{\mathcal{I}}$ and one of
the linear inequality constraints be active we will work with
$\mathbf{d}$, $\mathbf{g}$, $\mathbf{v}$, $\mathbf{y} \in
\mathbb{R}^{n-t-1}$ and the feasibility of iterations with respect
to linear constraints will be controlled by $\bar{\alpha}_k$.

The monotone line-search used in RCGD algorithm should satisfy
either of the standard Wolfe conditions or approximate Wolfe
conditions introduced in \cite{hager2005new}. In \cite{hager2005new}
a robust line search algorithm is suggested which tolerates the
finite accuracy precision of computer arithmetic very well. The
global convergence and superlinearly local convergence of the
presented RCGD algorithm can be directly followed from the
convergence theories of the original CGDESCENT algorithm discussed
in \cite{hager2005new}.

Following \cite{hager2006new}, the set of undecided indices is
defined as follows%
\[
\mathcal{U}(\mathbf{x}) = \{ i \in [1,n] \ | \ %
g_i(\mathbf{x}) \geqslant \|\mathbf{d}^1(\mathbf{x})\|^a \ %
\mathtt{and} \ %
\min\{x_i-l_i, u_i-x_i\} \geqslant \|\mathbf{d}^1(\mathbf{x})\|^b \}
\]
where $g_i(\mathbf{x})$ denotes $i$-th component of
$\nabla_\mathbf{x} f(\mathbf{x})$, $a \in (0,1)$ and $b \in (1,2)$
(e.g. $a=1/2$ and $b=3/2$). In fact, $\mathcal{U}(\mathbf{x})$
denotes the set of indices correspond to components of $\mathbf{x}$
for which the associated gradient component, $g_i(\mathbf{x})$, is
relatively large while $x_i$ is not close to either of $l_i$ or
$u_i$. When $\mathcal{U}(\mathbf{x})$ is empty, it implies that the
indices with large associated gradient component are almost
identified and it is preferable to perform unconstrained
optimization algorithm on the reduced space of free indices. This
naturally happens in the vicinity of the local solution.

Now let us to recall the HZ-ASA \cite{hager2006new} here. Using the
above mentioned SPG and  RCGD algorithm it is straightforward to use
HZ-ASA algorithm to solve knapsack problems. The outline of HZ-ASA
for knapsack problems is as follows:

\begin{center}%
\includegraphics[width=13.cm]{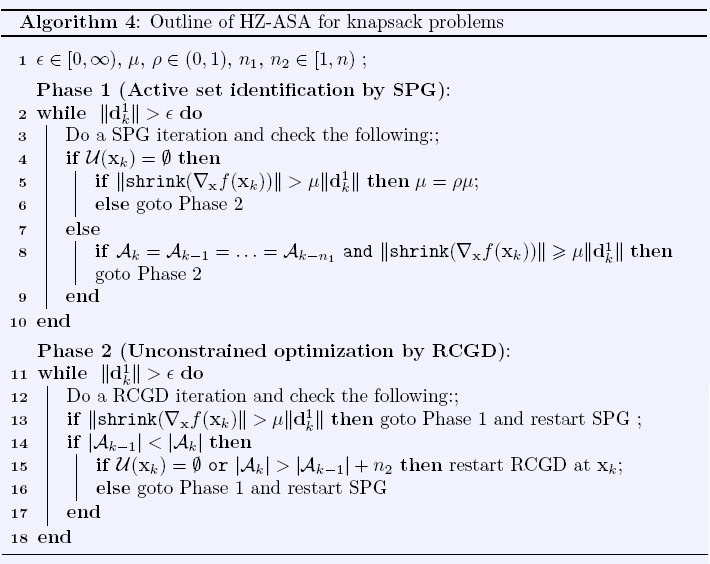}%
\end{center}%
\vspace*{10mm} %

The convergence theories stated in \cite{hager2006new} for HZ-ASA is
almost hold for algorithm 4. In the remaining part of this section,
we shall re-state results of \cite{hager2006new} for algorithm 4.

\begin{theorem} (global convergence)  %
\label{th:khzasa_gc}%
Let $\mathcal{L}$ be the level set defined by%
\begin{center}%
$\mathcal{L} = \{ {\bf x} \in \mathcal{D} : f ({\bf x}) \le f({\bf
x}_0)$ \}.%
\end{center}%
We assume the following conditions hold:%
\begin{itemize}%
    \item[{\rm G1.}]%
        $f$ is bounded from below in $\mathcal{L}$ and $d_{\max} =
        {\sup}_k \|{\bf d}_k\| < \infty$.%
    \item[{\rm G2.}]%
        If $\bar{\mathcal{L}}$ is the collection of ${\bf x}\in \mathcal{D}$
        whose distance to $\mathcal{L}$ is at most $d_{\max}$, then
        $\nabla f$ is Lipschitz continuous on $\bar{\mathcal{L}}$.%
\end{itemize}%
Then either algorithm 4 with $\epsilon=0$ terminates in a finite
number of iterations at a stationary point, or we have
$\liminf\limits_{k \to\infty} \| \textbf{d}^1(\textbf{x}_k)
\|_\infty = 0$.
\end{theorem}
\begin{proof}%
followed from theorem 4.1 of \cite{hager2006new}.
\end{proof}%

\begin{theorem} (global minimizer)  %
\label{th:khzasa_gc_convex}%
If $f$ is strongly convex and twice continuously differentiable on
$\mathcal{D}$, then the iterates $\mathbf{x}_k$ of algorithm 4 with
$\epsilon = 0$ converge to the global
minimizer of \eqref{eq:P}.%
\end{theorem}%
\begin{proof}%
see corollary 4.2 of \cite{hager2006new}.%
\end{proof}%

The following theorem results the local superlinear convergence rate
of algorithm 4 to a nondegenerate stationary point.

\begin{theorem} (nondegenerate local convergence)  %
\label{th:khzasa_lc_nondegenerate}%
If $f$ is continuously differentiable and the iterates
$\mathbf{x}_k$ generated by algorithm 4 with $\epsilon=0$ converge
to a nondegenerate (with respect to both of the bound and linear
constraints) stationary point $\mathbf{x}^*$, then after a finite
number of iterations, algorithm 4
performs only RCGD iterations without restarts.%
\end{theorem}%
\begin{proof}%
followed from theorem  5.1 of \cite{hager2006new}.%
\end{proof}%

The following theorems results the local superlinear convergence
rate of algorithm 4 to a stationary point (even degenerate
stationary point) under strong second-order sufficient optimality
condition.

\begin{theorem} (degenerate local convergence)  %
\label{th:khzasa_lc_degenerate_e}%
Assume $\mathcal{D} \stackrel{def}{=} \mathcal{D}_{\mathcal{E}}$ in
problem \eqref{eq:P} and $f$ is is twice-continuously
differentiable. If the iterates $\mathbf{x}_k$ generated by
algorithm 4 with $\epsilon=0$ converge to a stationary point
$\mathbf{x}^*$ satisfying the strong second-order sufficient
optimality condition, then after a finite number of iterations,
algorithm 4
performs only RCGD iterations without restarts.%
\end{theorem}%
\begin{proof}%
followed from theorem 5.7 of \cite{hager2006new}.%
\end{proof}%

\begin{theorem} (degenerate local convergence)  %
\label{th:khzasa_lc_degenerate_i}%
Assume $\mathcal{D} \stackrel{def}{=} \mathcal{D}_{\mathcal{I}}$ in
problem \eqref{eq:P} and $f$ is is twice-continuously
differentiable. If the iterates $\mathbf{x}_k$ generated by
algorithm 4 with $\epsilon=0$ converge to a stationary point
$\mathbf{x}^*$ satisfying the strong second-order sufficient
optimality condition and  $\mathbf{x}^*$ is not degenerate with
respect to the linear inequality constraint, i.e., the optimal value
of lagrange multiplier corresponding to the active linear constraint
is nonzero, then after a finite number of
iterations, algorithm 4 performs only RCGD iterations without restarts.%
\end{theorem}%
\begin{proof}%
followed from theorem 5.7 of \cite{hager2006new}.%
\end{proof}%

\begin{remark}
There is a simple strategy to virtually benefit from the local
superlinear convergence rate of algorithm 4 in the case of
degeneracy of the minimizer with respect to linear inequality
constraints. Since the different states of the active set of
\eqref{eq:P} with respect to linear constraints includes only three
situtions, it is possible to  solve \eqref{eq:P} sequentially three
times with algorithm 4 using its strong local convergence. Assume,
$\mathbf{x}_m^*$, $\mathbf{x}_l^*$ and $\mathbf{x}_u^*$ are
solutions of \eqref{eq:P} by algorithm 4 without considering the
linear constraint in RCGD, considering the linear equality
constraint with $b=b_l$ in RCGD and considering the linear equality
constraint with $b=b_u$ in RCGD respectively. Note that we can first
solve for $\mathbf{x}_m^*$ to possibly avoid further computation. If
$\mathbf{x}_m^*$ satisfies the bilateral inequality constraint, it
is equal to a local solution of problem and the computation will be
discarded (of course in this case there is no degeneracy with
respect to linear constraints). Otherwise, the local solution of
problem is one of $\mathbf{x}_l^*$ or $\mathbf{x}_u^*$ which
correspond to the smaller objective function value. Therefore, this
strategy finds a local solution of problem by solving three
sequential superlinearly convergent optimizations. However, as the
cost of computation is tripled by this strategy, the overall rate of
convergence can be sub-linear in the worst conditions.
\end{remark}


\section{Birgin-Martinez active set algorithm for knapsack problems}%
\label{sec:BMASA}%

In \cite{birgin2002large} an active set  algorithm is introduced to
solve large-scale bound constrained optimization problems. The
general outline of this algorithm is very similar to that of Hager
and Zhang; and differences are mainly related to some technical
issues. It's worth mentioning that this algorithm does not posses a
superlinear local rate of convergence like Hager-Zhang algorithm,
and in the best conditions the local rate of convergence will not be
better than linear. However, our numerical experiences showed that
the performance of this algorithm is very promising solving
large-scale real world problems (e.g. when using accurate functional
value and its gradient is not economic or possible).

Like Hager-Zhang active set  algorithm, Birgin-Martinez active set
algorithm is a two-phase algorithm: a gradient projection step to
explore the working set and an unconstrained optimization solver on
the space spanned by free variables. In \cite{andretta2010partial}
this algorithm was extended to solve linearly constrained
optimization problems. Due to its expensive projection and
null-space steps, this algorithm is not feasible for large-scale
problems. Therefore, our algorithm is somehow an adaptation of
\cite{andretta2010partial} algorithm to solve knapsack problems. To
do this job, one case follows the same procedure as described in the
previous section. More clearly, this algorithm includes a general
projection onto constraint space which should be replaced by our
problem-specific projection step. Moreover, this algorithm uses a
general null-space manipulation of active constraints which again
should be replaced by our efficient problem-specific null-space
method.

Our adapted algorithm here will be almost identical to Algorithm 3.1
of \cite{andretta2010partial} with this difference that the partial
projection step in \cite{andretta2010partial} should be replaced by
our exact projection step (the parameter $\delta$ in this algorithm
should be replaced by $\infty$). The convergence theories of our
algorithm will be in fact the simplified version of that of
\cite{andretta2010partial}. For more details, interested readers are
referred to \cite{birgin2002large,andretta2010partial} and also our
Fortran implementation of method which is freely available from web.


\section{Application to topology optimization problems}%
\label{sec:topopt}%

In this section we apply the introduced method to solve topology
optimization problems, cf. \cite{bendsoe2003topology}. As a model
problem we will solve the problem already considered in
\cite{donoso2006numerical}. The description of this problem is as
follows: considering two isotropic conducting materials, with
thermal conductivities $k_\alpha$ and $k_\beta$, $0 < k_\alpha <
k_\beta$, in a simply connected design domain $\Omega \subset
\mathbb{R}^d (d = 2, 3)$; the goal is mixing these materials with a
fixed ratio in $\Omega$, such that the total temperature gradient in
$\Omega$ is minimized under the thermal load $f \in L^2(\Omega)$, more precisely:%

\begin{equation}%
  \nonumber
    (P) \quad %
    \left\{%
\begin{array}{rll}
      \arg \min_{w}  J(w) =  &\frac 12 \int_{\Omega} \nabla \theta \cdot \nabla \theta \ d\textbf{x},
      &\\\\
      \texttt{subject to:} \\\\
      - \nabla \cdot (k(w) \nabla \theta) = & f(\mathbf{x})   &
    \texttt{in} \quad \Omega \\
    \theta(\mathbf{x})  = & \theta_0(\mathbf{x})   &
    \texttt{on} \quad \partial\Omega  \\ %
    k(w)    = & w k_\beta  + (1-w) k_\alpha             &
    \texttt{in} \quad \Omega  \\ %
    \int_{\Omega} w \ d \textbf{x}   = & R |\Omega|, \quad 0<R<1, & 0  \leqslant w \leqslant
    1
\end{array}%
\right.%
\end{equation}%
\\
where $\theta \in H^1(\Omega)$ is the state variable and $w \in
L^2(\Omega)$ is the control parameter (topology indicator field). By
straightforward derivation, the first order optimality condition for
problem P can be written in the following form:

\begin{equation}%
  \nonumber
    (OC) \quad %
    \left\{%
\begin{array}{rll}
      - \nabla \cdot (k(w) \nabla \theta) = & f(\mathbf{x})   &
    \texttt{in} \quad \Omega \\
    \theta(\mathbf{x})  = & \theta_0(\mathbf{x})   &
    \texttt{on} \quad \partial\Omega  \\ %
    - \nabla \cdot (k(w) \nabla \eta) = & -\nabla\cdot(\nabla \theta)  &
    \texttt{in} \quad \Omega \\
   \eta(\mathbf{x})  = & 0   &
  \texttt{on} \quad \partial\Omega \\  %
    k(w)    = & w k_\beta  + (1-w) k_\alpha             &
    \texttt{in} \quad \Omega \\  %
    \mathcal{P}_{\mathcal{D}_\mathcal{I}}\big(G(\mathbf{x}) \big)    = & 0 &
    \texttt{in} \quad \Omega  \\%
    G(\mathbf{x})    = & -(k_\beta - k_\alpha) \nabla \theta \cdot \nabla \eta &
    \texttt{in} \quad \Omega  \\%
\end{array}%
\right.%
\end{equation}%
\\
where $\eta \in H^1_0(\Omega)$ is the adjoint state, $G$ is the
$L^2$ gradient of the objective functional with respect to $w$ and
$\mathcal{P}_{\mathcal{D}_\mathcal{I}}(u)$ denotes the $L^2$
projection of function $u$ onto the admissible set
$\mathcal{D}_\mathcal{I}$,
\[
\mathcal{D}_\mathcal{I} = \{ w \in L^2(\Omega) \ | \  \int_\Omega
w(\mathbf{x}) d \mathbf{x}  = R |\Omega|, \quad 0<R<1, \quad 0
\leqslant w \leqslant 1 \}
\]

Consider the finite-dimensional counterpart of problem (P) which is
resulted by discretization of $\Omega$ into an $n$ control volumes.
Also assume that the state variables and design parameter are
defined at the center of each control volume. Under these
assumptions, the admissible design domain, $\mathcal{D}_\mathcal{I}$
forms a simplex in $\mathbb{R}^n$ which is identical to the knapsack
constraint. If we solve the state PDE at every stage of optimization
(for a specific value of $w$), then the above optimization problem
will be equivalent to problem \eqref{eq:P}. Note that vector
$\mathbf{a}$ in \eqref{eq:deq} is equivalent to the volume of each
control in this model problem.

Since it is not easy to construct the above topology optimization
problem such that it ensures pre-requirement of extended Hager-Zhang
active set algorithm, the extended Birgin-Martinez active set
algorithm will be used here. Moreover, it is worth mentioning that
in all of our experience with these algorithm (solving similar
optimization problems), the extended Birgin-Martinez performs
superior. We think that main reasons for this observation is the
inaccuracy of objective function and its gradient in our experience
besides to our failure to adjust the algorithm's control
appropriately.

To do numerical experiment we solve problem (P) in two and three
dimensions for $\Omega = [0,1]^2$ and  $\Omega = [0,1]^3$
respectively. The spatial domain is divided into $127^3$ and $31^3$
control volumes in two and three dimensions respectively. In all of
experiments here $k_\alpha = 1$, $k_\beta = 2$, $f(\mathbf{x}) = 1$
and $R = 0.4$. The governing PDE are solved by cell centered finite
volume method using central difference scheme and related system of
linear equations are solved by preconditioned conjugate gradient
method with relative convergence threshold $1.e-20$. As an stopping
criteria, the optimization is discarded when the relative variation
of topology becomes bellow $1.e-3$. Notice that using finite volume
method we do not observe any topological instability problem
(checkerboard pattern) which is common in topology optimization
problems using finite element method.

Results of our numerical experiments includes the variation of the
objective functional during optimization cycles and final resulted
topology ($w$-field), are shown in figures \ref{fig:obj_hist} and
\ref{fig:final_top}. Plots obviously show the success of the
presented algorithm to solve topology optimization problems. Note
that in our procedure the PDE constraint is satisfied upto the
accuracy of method used to solve PDE and knapsack constraint is
satisfied upto the machine precision. It is worth mentioning that we
already applied the presented method successfully to vast families
of topology optimization problems which are not included here due to
the space constraint.

\begin{figure}[ht]%
\begin{center}%
\includegraphics[width=13.cm]{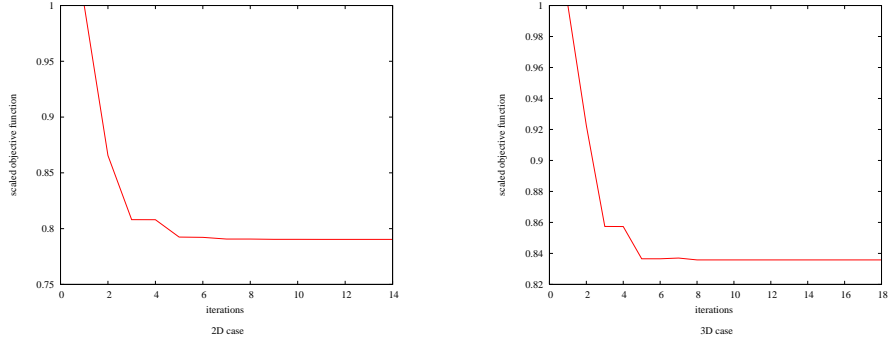}%
\caption{Variation of scaled objective function during optimization
cycles for 2D (left) and 3D (right) examples.
}%
\label{fig:obj_hist}%
\end{center}%
\end{figure}%
\begin{figure}[ht]%
\begin{center}%
\includegraphics[width=14.cm]{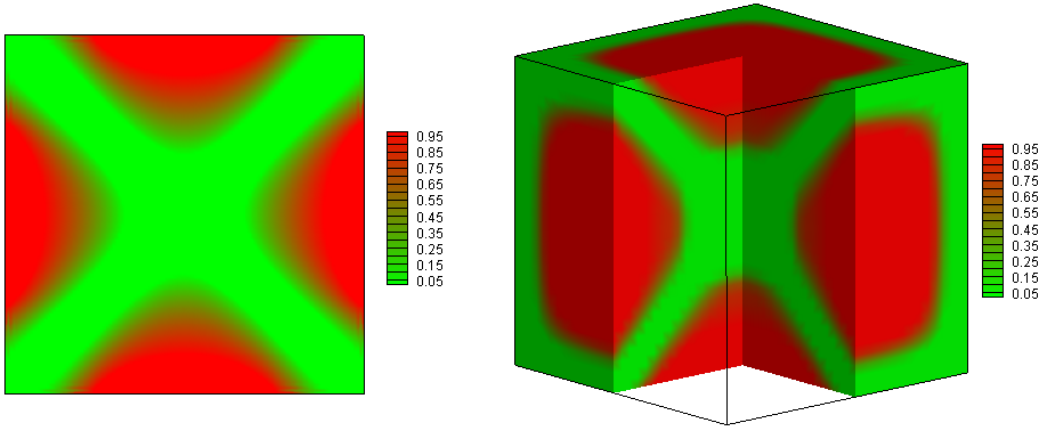}%
\caption{Final distribution of material field, $w$, for 2D (left)
and 3D (right) examples.
}%
\label{fig:final_top}%
\end{center}%
\end{figure}%
%


\section{Summary}%
\label{sec:summary}%

In the present study, a general framework is suggested to adapt
bound-constrained optimization solvers to solve optimization
problems with bounds and an additional linear constraint (in the
form of either equality or bilateral inequality). This framework
includes two main ingredients which are the projection of an
arbitrary point onto the feasible set and null-space manipulation of
the related linear constraint. The implementation of the method with
specific focus on the Hager-Zhang active set algorithm
\cite{hager2006new} is discussed in details. Moreover, we briefly
discuss the adaptation of Birgin-Martinez active set algorithm
\cite{birgin2002large} to there problems which is structurally very
similar to Hager-Zhang active set algorithm.

It seems that, following the presented approach, it should be
possible to perform such an extension for alternative
box-constrained optimization methods. Notice that this extension can
be much simpler and straightforward for some box-constrained
solvers. For instance, using just the projection operator introduced
in this study, it is natural to extend Newton method of
\cite{lin1999newton} or affine-scaling interior-point Newton method
of \cite{heinkenschloss1999superlinear} for this purpose without
further difficulties.


\section*{Acknowledgements}
The author would like to thanks Hongchao Zhang, Marina Andretta and
Ernesto Birgin for constructing comments.


\bibliographystyle{plain} 
\bibliography{biblio}%


\end{document}